\documentclass[10pt,draft,oneside]{amsart}
\usepackage{enumerate}
\usepackage{amsfonts}
\usepackage{amsmath,amsthm,amssymb}

\newcommand{\comment}[1]{}

\newtheorem{theorem}{Theorem}

\newtheorem{proposition}[theorem]{Proposition}
\newtheorem{lemma}[theorem]{Lemma}
\newtheorem{corollary}[theorem]{Corollary}
\newtheorem*{proposition*}{Proposition}
\newtheorem*{theorem*}{Theorem}
\theoremstyle{definition}
\newtheorem*{definition*}{Definition}

\newtheorem*{question}{Question}

\newcommand\con{\qopname\relax o{conv}}

\newcommand{\NN}{\mathbb N}
\newcommand{\RR}{\mathbb R}

\newcommand\PP{\mathcal{P}}
\newcommand\Pn{\mathcal{P}_n}
\newcommand\Pm{\mathcal{P}_m}

\newcommand\spn{\mathop{\mathrm{span}}}

\newcounter{step}
\def\resstep{\setcounter{step}{0}}
\def\step#1{\stepcounter{step}\noindent\textbf{Step \thestep.}\hskip1ex}

\begin{document}
\title{Positive bases in spaces of polynomials}


\author[B. Farkas]{B\'alint Farkas}
\address[B\'alint Farkas]{Technische Universit\"{a}t Darmstadt
\newline Fachbereich Mathematik, AG4
\newline Schlo\ss{}gartenstra\ss{}e 7, D-64289, Darmstadt, Germany.}
\email{farkas@mathematik.tu-darmstadt.de}

\author[Sz. Gy. R\'ev\'esz]{Szil\'ard Gy. R\'ev\'esz}
\address[Szil\'ard Gy. R\'ev\'esz]{A. R\'enyi Institute of Mathematics
\newline Hungarian Academy of Sciences,
\newline Budapest, P.O.B. 127, 1364 Hungary.}
\email{revesz@renyi.hu}

\date{\today}


\subjclass[2000]{Primary 41A17. Secondary 30E10, 41A44.}

\keywords{Positive polynomials, positive bases, Bernstein-Lorentz representation, Bernstein-Lorentz degree}


\let\oldfootnote\thefootnote
\def\thefootnote{}

\thanks{The 2nd named author was supported in part by the Hungarian
National Foundation for Scientific Research, Project \#s T-049301,
T-049693 and K-61908.}

\thanks{This work was accomplished during the 2nd author's stay in Paris
under his Marie Curie fellowship, contract \# MEIF-CT-2005-022927.}




\begin{abstract}For a nonempty compact set $\Omega\subseteq \RR$ we
determine the maximal possible dimension of a subspace
$X\subseteq\Pm(\Omega)$ of polynomial functions over $\Omega$ with
degree at most $m$ which possesses a \emph{positive basis}. The
exact value of this maximum depends on topological features of
$\Omega$, and we will see that in many of the cases $m$ can be
achieved. Whereas only for low $m$ or finite sets $\Omega$ is it
possible that we have a subspace $X$ with positive basis and with
$\dim X=m+1$. Hence there is no $\Omega$ for which a positive basis
exists in $\Pm$ for all $m\in\NN$.
\end{abstract}

\maketitle

\section{Introduction}

Consider the interval $I=[-1,1]$ and the space $\Pn$ of algebraic
polynomials with degree at most $n\in\NN$ over $I$. In many
problems, e.g., in approximation theory, it is desirable to
represent a given polynomial $p\in \bigcup_n \Pn=:\PP$ in terms of
positive linear combinations of preliminary given, \emph{positive}
polynomials (positive means here and in the following that the
polynomial is pointwise nonnegative). For example consider the
Bernstein polynomials
\begin{equation*}
e_{nk}(x):=(1-x)^k(1+x)^{n-k}\quad\mbox{for $k=0,\dots,n$},
\end{equation*}
which are evidently positive. The following is classical.
\begin{proposition*} The system $E_n:=\{e_{n0},e_{n1},\dots,e_{nn}\}$ is a basis of $\Pn$.
\end{proposition*}
\begin{proof} For the number of elements $\# E_n= n+1$ equals $\dim \Pn$, it suffices to show
that $E_n$ generates $\Pn$. For this purpose, we can represent the
monomials $x^k$ as
\begin{align*}
x^k& =\Bigl(\tfrac12 [(1+x)-(1-x)]\Bigr)^k \\& =\Bigl(\tfrac12
[(1+x)-(1-x)]\Bigr)^k\Bigl(\tfrac12 [(1+x)+(1-x)]\Bigr)^{n-k},
\end{align*}
which, when expanded, yields a desired representation in terms of
$e_{nk}$. From this the assertion follows.
\end{proof}
Now back to the original question.  Does every positive element
$p\in\PP$ have a positive representation
\begin{align}\label{eq:posrep}
p(x)&=\sum_{k=0}^N a_k (1-x)^k(1+x)^{N-k} = \sum_{k=0}^N a_k
e_{Nk}(x) \quad \mbox{with $a_0,a_1,\dots,a_N\geq 0$}
\end{align}
for some $N\in\NN$? The affirmative answer is due to Bernstein
\cite{bernstein:1952}, see also \cite[vol. II p. 83, Aufgabe
49]{polya/szego:1925}.

\begin{theorem*}[Bernstein]\label{th:Bernstein} Every polynomial
positive on $[-1,1]$ has a positive representation
\eqref{eq:posrep}.
\end{theorem*}

However, the precise truth is that the degree $N$ of the
representation \eqref{eq:posrep} is in general not equal to the
ordinary algebraic degree $n=\deg p$ of $p$. The minimal such $N$ is
called the \emph{Lorentz-degree} $L(p)$ of $p$.  For a positive
polynomial $p$ Bernstein's theorem is thus equivalent to
$L(p)<\infty$. An example for the use of positive representations in
approximation theory is the improvement of Bernstein, Schur and
Markov type inequalities for polynomials, which can be deduced by
replacing the ordinary degree by the Lorentz degree; for the details
see \cite{revesz:2006} and the references therein.

Among the various known estimations of the Lorentz degree, the
first, and probably the simplest, is the following.
\begin{theorem*}[Lorentz \cite{lorentz:1963}; see also \cite{scheick:1972}]\label{th:lorentz}
Assume that the positive polynomial $p$ does not vanish in the unit
disk. Then $L(p)=\deg p$.
\end{theorem*}

On the other hand, it is well-known that for $p\in\Pn$ positive the
Lorentz degree can be arbitrarily large. In other words, $\spn
\{e_{N0},e_{N1},\dots, e_{NN}\}$ does not contain $\Pn$ for any
$N\in \NN$. Thus in particular $E_N$ is not a positive basis for
$\PP_N$. A \emph{positive} basis in $\Pn$ is by definition  a basis
in $\Pn$ such that each positive element $p\in \Pn$ is a positive
linear combination of the basis elements. Positive bases have been
extensively studied by Polyrakis  in \cite{polyrakis:1987} and
\cite{polyrakis:1996} (see also \cite{polyrakis:1994,
polyrakis:1999} and \cite{polyrakis:2003}).

Little surprising is the following result of Polyrakis (see
Proposition 3.2, Corollary 3.3 and Example 4.2 of
\cite{polyrakis:1996}).
\begin{proposition}[Polyrakis] The space $\Pn$  does not have a positive basis at all.
\end{proposition}

Actually this negative result is the starting point of this paper.
We study the following problem. For  $X\subseteq \Pm$  a subspace
with positive basis, we are interested in how large the dimension
$\dim X$ can be. As observed above the maximal possible dimension
must be less than or equal to $m$, the exact answer will be given in
Section \ref{sec:interval}.

We also consider the more general case when the polynomials are
considered over a given nonempty compact subset
$\Omega\subseteq\RR$. For the determination of maximal dimension of
a subspace of $\Pm$ with positive basis we have to take into account
some topological-geometrical features of the set $\Omega$.  This
will be carried out in Section \ref{sec:genrealline}.

We begin with some preparation.
\subsection{Subspaces of $C(\Omega)$ with positive
bases}\label{subs:intr} We denote by $\Omega$ a nonempty compact set
in $\RR$, and by $C(\Omega)$ the space of real valued, continuous
functions defined on $\Omega$. We consider $C(\Omega)$ (partially)
ordered by the pointwise ordering. The notation $\Pn(\Omega)$ stands
for the space polynomials (polynomial functions) $p$ with $\deg
p\leq n$ restricted to the set $\Omega$. The ordering of the
polynomials is the one inherited from $C(\Omega)$.  As $\Omega$ is
fixed once for all, we will simple use the notation $\Pn$ instead of
$\Pn(\Omega)$, as well $\PP$ for $\bigcup_{n\in\NN} \Pn(\Omega)$.

Let $\{b_j\}$ be a (finite or infinite) sequence of $C(\Omega)$. If
$t$ is a point of $\Omega$ and there exists $k\in\mathbb{N}$ such
that $b_k(t)\neq 0$ and $b_j(t)=0$ for each $j\neq k$, then we call
that $t$ an $k$-\emph{node} (or simply a \emph{node}) of the
sequence $\{b_j\}$. If for each $k$ there exists an $k$-node $t_k$
of $\{b_j\}$, we shall say that $\{b_j\}$ is a \emph{sequence} of
$C(\Omega)$ \emph{with nodes} and also that $\{t_j\}$ is a
\emph{sequence of nodes} of $\{b_j\}$.

The following results form the background for our paper  (see
Theorem 2.1 and Propositions 2.2 and 2.4 of \cite{polyrakis:1996} or
Theorem 7 of \cite{polyrakis:2003}) and helps to find node systems.
\begin{theorem}\label{th:generalprop}
Let $E$ be a subspace of $C(\Omega)$ and $\{b_r\}$ be a sequence of
$E$ consisting of positive functions.
\begin{enumerate}[i)]
\item
If $\{b_r\}$ is a positive basis of $E$, then  for each $k$ there
exists a sequence $\{\omega_{m\nu}\}$ of $\Omega$
 such that for each $\nu\in\mathbb{N}$ we have
\begin{equation*}
0\le\sum_{i=1,i\neq k}^\nu\frac{b_i(\omega_{k\nu})}{
b_k(\omega_{k\nu})}<\frac{1}{\nu}.
\end{equation*}
Therefore
$\lim_{\nu\to\infty}\frac{b_i(\omega_{k\nu})}{b_k(\omega_{k\nu})}=0$
for each $i\neq k$.
\item
If $E$ is  an \emph{$n$-dimensional} subspace of $C(\Omega)$ and the
sequence $\{b_r\}$ consists of $n$ vectors $b_1, b_2, ..., b_n $,
then the converse of $(i)$ is also true. That is, if for each $1\le
k\le n$ there exists a sequence $\{\omega_{k\nu}\}$ of $\Omega$
 satisfying
\begin{equation*}
\lim_{\nu\to\infty}\frac{b_i(\omega_{k\nu})}{b_k(\omega_{k\nu})}=0\;\hbox{
for each $i\neq k$},
\end{equation*}
then $\{b_1,...,b_n\}$ is a positive basis of $E$.
\end{enumerate}
\end{theorem}

Let $a=\min\Omega$, $b=\max\Omega$. Unless $\Omega$ is finite, which
leads to an equivalent question of considering the space $\RR^K$,
(where $K:=\# \Omega$ the number of elements in $\Omega$), we have
$\dot{\Omega} \ne \emptyset$ (where $\dot{H}$ stands for the set of
limit points of the set $H$). Unless otherwise stated, this will be
assumed  from now on. For the finite case of $\RR^K$, see for
example \cite{polyrakis:1994} (see also Proposition
\ref{p:sidecases}). Suppose that $X$ is a subspace of $\Pm$ with
positive basis $\{b_1,b_2,...,b_n\}$. For each $k=1,2,...,n$ let
$\omega_{k\nu}\longrightarrow t_k$ be a sequence of $\Omega$ with
the properties formulated above in Theorem \ref{th:generalprop},
that is
\begin{equation}\label{eq:e3}
\lim_{\nu\to\infty}\frac{b_i(\omega_{k\nu})}{b_k(\omega_{k\nu})}=0\;\hbox{
for each $i\neq k$}.
\end{equation}
Note that \eqref{eq:e3} entails $t_k\neq t_i$ for each $i\neq k$.
Then we can write with some $u_i(t)\in\Pm$, not vanishing at any of
the points $t_j$, that
\begin{equation*}
b_i(t)= \prod_{k=1}^n (t-t_k)^{r_{ik}}u_{i}(t),
\end{equation*}
where $r_{ik}$ is the order of $t_k$ as root of the polynomial
$b_i$. (If $t_k$ is not a root of $b_i$, then $r_{ik}=0$.) Moreover,
again by \eqref{eq:e3} we obtain that $b_i$ vanishes in a strictly
greater order at $t_k$ than $b_k$ for $i\neq k$. That is we have
$r_{ik}>r_{kk}$ (the latter may be $0$).

Using the above facts, we are going to estimate the degrees of $b_k$
from \emph{below} in terms of $n$, which, conversely, will give us
an \emph{upper} estimate of $n$ in terms of the degrees of $b_k$.
Thus we can estimate the dimension of a subspace $X\subset
\mathcal{P}_m$ with positive basis.

\section{Positive bases in subspaces of polynomials on intervals}\label{sec:interval}

Assume now that $\Omega=[a,b]$ and that $X\subseteq \mathcal{P}_m$
is subspace with positive basis $b_1,b_2,\dots,b_m$. By the
consequence of Polyrakis's result we have $t_1,t_2,\dots, t_n\in
[a,b]$ such that
\begin{equation}\label{eq:factor}
b_i(t)= \prod_{k=1}^n (t-t_k)^{r_{ik}}u_{i}(t),
\end{equation}
where $r_{ik}$ are \emph{even} natural numbers whenever the
corresponding $t_k$ satisfies $a<t_k<b$, because otherwise $b_k$
would change its sign at $t_k$.

As explained at the end of Section \ref{subs:intr}, we want to
estimate the degrees of $b_k$ from below. Let us do it for the sake
of clarity in a detailed manner. We therefore distinguish three
cases depending on whether $\{t_1,t_2,\dots,t_n\}\cap \{a,b\}$ has
$0,1$ or $2$ elements.

\vskip1ex\noindent Consider first the case $a=t_1<t_2<\cdots<t_n=b$.
Then for $k=1,n$ we obtain the following from \eqref{eq:factor}
\begin{align*}
\deg b_1&= \deg u_1+\sum_{i=1}^n r_{1i}\geq
0+r_{1n}+2\sum_{i=2}^{n-1}r'_{1i} +r_{nn}\geq 1+2(n-2)+0=2n-3,\\
\deg b_n&= \deg u_n+\sum_{i=1}^n r_{ni}\geq
0+r_{11}+2\sum_{i=2}^{n-1}r'_{ni} +r_{n1}\geq 0+2(n-2)+1=2n-3.
\end{align*}
For $1<k<n$ we have
\begin{equation*}
\deg b_k= \deg u_k+\sum_{i=1}^n r_{ki}\geq0+ r_{k1}+ 2\sum_{i=2\atop
i\neq k}^{n-1}r'_{ki} +r_{kk}+r_{kn}\geq 1+2(n-3)+0+1=2n-4.
\end{equation*}
Where in all of these estimates, we used that $r_{kk}\geq 0$,
$r_{ki}>0$ for $k\neq i$ and $r_{ki}=2r'_{ki}>0$ if $1<k<n$, which
again follow from the considerations at the end of the Section
\ref{subs:intr}

Similar arguments can be repeat in the case, when no $t_l$ coincides
with $a$ or $b$, so all the exponents $r_{ki}$ with $k\neq i$ must
be even and positive. Thus we can conclude from \eqref{eq:factor}
\begin{equation*}
\deg b_k= \deg u_k+\sum_{i=1}^nr_{ki}\geq 0+2\sum_{i=1\atop i\neq
k}^nr'_{ki}+r_{kk}\geq 2(n-1)+0=2(n-2).
\end{equation*}

Finally, a very same argumentation can be executed also for the
cases when among the endpoints $\{a,b\}\cap \{t_1,t_2,\dots ,t_n\}$
has exactly $1$ element; in that case we find again
$\max_{k=1,2,\dots,n} d_k \geq 2n-2$.

Therefore we conclude the inequality
\begin{equation}\label{eq:5}
\max_{k=1,\dots,n} \deg b_k\geq 2n-3
\end{equation}
for all possible cases with respect to $\# \{a,b\}\cap
\{t_1,t_2,\dots t_n\}$ being $0$, $1$ or $2$. As $b_i\in\Pm$, the
lower bound \eqref{eq:5} for the degree has to match with the upper
bound $m$ of the degree. Therefore the maximal possible value for
$n$ is
\begin{equation*}
n=\left[\frac{m+3}{2}\right].
\end{equation*}
An easy construction shows this to be sharp. Let $n=[\frac{m+3}{2}]$
and take $n$ different points  $a=t_1<t_2<\dots <t_n=b$ of
$\Omega=[a,b]$ and define
\begin{align}\label{eq:biforint}
\nonumber b_1(t)&=(t_n-t)\prod_{i=2}^{n-1}(t-t_i)^2,&& \qquad (t_1=a)\\
b_k(t)&=(t-t_1)(t_n-t)\prod_{i=2 \atop i\neq k}^{n-1}(t-t_i)^2,&&
\qquad (k=2,\dots,n-1)\\
\nonumber b_n(t)&=(t-t_1)\prod_{i=2}^{n-1}(t-t_i)^2,&& \qquad(t_n=b)
\end{align}
and take $X$ to be the subspace of $\Pm$ generated by the $b_k$'s.
Then $\{b_1,b_2,\dots,b_n\}$ is a positive basis of $X$, because if
we suppose that $p=\sum_{i=1}^n\lambda_ib_i$ we have that $p(t_i)
=\lambda_ib_i(t_i)$, therefore $p$ is positive if and only if
$\lambda_i\geq 0$ for each $i$. So we have the following
\begin{theorem} The maximal possible dimension of a subspace $X$ of $\Pm$ with a
positive basis is $n=[\frac{m+3}{2}]$. If $X\subseteq \Pm$ is such
subspace with positive basis $\{b_1,b_2,...,b_n\}$, then there exist
a node system  $t_1,t_2,...,t_n\in [a,b]$ with $t_i\neq t_j$ for
each $i\neq j$ such that the basis elements are either of the form
\eqref{eq:biforint} with $a=t_1<\dots<t_n=b$, or only in case of
even $m$, is one of the following forms: either with
$a=t_1<\dots<t_n\leq b$
\begin{equation*}
b_1(t)=\prod_{k=2}^{n}(t-t_k)^2, \quad\textrm{and}\quad
b_i(t)=(t-t_1)\prod_{k=2 \atop k\neq i}^{n}(t-t_k)^2, \quad
(i=2,\dots,n);
\end{equation*}
or with $a\leq t_1<\dots<t_n=b$
\begin{equation*}
b_i(t)=(t_n-t)\prod_{k=1 \atop k\neq i}^{n-1}(t-t_k)^2, \quad (1\leq
i \leq n-1) \quad \textrm{and} \quad
b_n(t)=\prod_{k=1}^{n-1}(t-t_k)^2;
\end{equation*}
or with $a\leq t_1<\dots<t_n\leq b$
\begin{equation*}
b_i(t)=\prod_{k=1 \atop k\neq i}^{n}(t-t_k)^2, \quad (1\leq i \leq
n).
\end{equation*}

\end{theorem}


\section{Positive basis of polynomials on compact sets of $\RR$}\label{sec:genrealline}

In this section we consider $\Omega\subset\RR$ compact sets. Our aim
is to determine the minimum of possible maximal degrees of systems
of basis functions $b_j$ ($j=1,\dots,n$) of some subspace
$X\subset\Pm$ on $\Omega$. The answer will be more delicate than in
the case of an interval.

We will always write $a:=\min \Omega$ and $b:=\max \Omega$, i.e.,
$\con \Omega =[a,b]$.  We also need to introduce  some auxiliary
terminology.

\begin{definition*}
Let $t=(t_j)_{j=1}^n$, ordered naturally as $a\leq
t_1<t_2<\dots,t_n\leq b$, be a node system for the basis functions
$b_j$, ($j=1,\dots,n$) of the subspace $X$ in $\Pm$. For technical
reasons we take now $t_0:=-\infty$ and $t_{n+1}:=+\infty$, and we
define the \emph{type} of the node sequence $t$ by a sequence
$\omega:=(\omega_j)_{j=0}^n$ of $0$'s and $1$'s, of length $n+1$, so
that $\omega_j=0$ if $\Omega\cap (t_{j},t_{j+1})= \emptyset$, and
$\omega_j=1$ if $\Omega\cap (t_{j},t_{j+1})\ne \emptyset$.
Therefore, there is $x_j\in\Omega\cap (t_{j},t_{j+1})$ if and only
if $\omega_j=1$. The notation $x(t)$ or $x_j$ will always refer to
such a sequence from now on.
\end{definition*}

As in \eqref{eq:factor} for a given system of nodes $t$ we can write
\begin{equation*}
b_j(x)=\prod_{i=1}^n (x-t_i)^{r_{ij}} u_j(x) \qquad (j=1,\dots,n),
\end{equation*}
with $r_{ij}\geq 1$ for all $j\ne i$ and $u_j$ having no zeroes at
points of $t$.

Our task is to determine, in terms of geometrical-topological
features of the set $\Omega$, the minimal $m$ so that all $b_j$ are
contained in $\Pm$. This means that we consider all possible
positive basis $b_1,b_2,\dots,b_n$  and  minimize the occurring
maximal degree of basis elements. Obviously, it is equivalent to
considering all basis $b_1,b_2,\dots,b_n$ and then all nodes $t$ for
them, or conversely, all node systems $t$ and then all basis systems
$b_1,b_2,\dots,b_n$ having these as nodes. Therefore, our goal is
the determination of
\begin{align}\label{eq:dnomegadef}
d_n(\Omega)&:=\min \Bigl\{\max_{j=1,\dots,n} \deg b_j ~~:~b_1,\dots b_n ~~ \textrm{is a positive basis}\Bigr\}\notag \\
&=\min \Bigl\{\max_{j=1,\dots,n} \deg b_j ~~:~(b_j)_{j=1}^n ~~\textrm{is a p.b. with nodes}~~t\in\Omega^n\Bigr\}\\
&=\min_{t\in\Omega^n} \min \Bigl\{\max_{j=1,\dots,n} \deg b_j
~~:~~(b_j)_{j=1}^n ~~\textrm{is a p.b. with nodes}~~t\Bigr\}\notag
\end{align}
Before computing $d_n(\Omega)$, we consider the closely related,
auxiliary question of finding
\begin{equation}\label{eq:bdegree}
\min \Bigl\{\deg b  ~:~ 0\leq b|_{\Omega},~ b(x)=\prod_{i=1}^n
(x-t_i)^{r_{i}} u(x), ~~ r_i\geq 1 ~(i=1,\dots,n)\Bigr\}
\end{equation}
for a fix system of nodes $t\subset\Omega$. Simplifying this we
consider a modified problem, with a particular subset
$\Omega':=t\cup x(t)$ replacing $\Omega$ in the nonnegativity
condition:
\begin{equation}\label{eq:bdegreeprime}
\min \Bigl\{\deg b ~:~ 0\leq b|_{\Omega'},~ b(x)=\prod_{i=1}^n
(x-t_i)^{r_{i}} u(x), ~~ r_i\geq 1
~(i=1,\dots,n)\Bigr\};%
\end{equation}
Or just as a function of $\omega$:
\begin{align} \label{eq:tauomega}
\tau(\omega):=\min\Bigl\{\deg p :& ~p\not\equiv 0,~~p(t_i)=0
~(i=1,\dots n),\notag \\ &  ~ p(x_j)\geq 0\mbox{ for all $j$ with
$\omega_j=1$} \Bigr\}.
\end{align}
Note that according to the requirement $p(t_i)=0$, we always must
have $\tau(\omega)\geq n$. Although in \eqref{eq:bdegreeprime} the
quantity seems to depend on particular values of $t_i$ and $x_j$, in
the reworded formulation \eqref{eq:tauomega} we have already
retreated in notation to the mere mentioning of $\omega$. That the
quantities in \eqref{eq:bdegreeprime} and \eqref{eq:tauomega} are
equal is an easy argument: here we spare the reader the details, as
this will also follow from our considerations below, see Corollary
\ref{c:omegaandprime} . First we determine the value of
$\tau(\omega)$ for a given $\omega=\omega(t)$. We need the
following.
\begin{definition*}A \emph{block} of zeros or ones in $\omega$ is a
maximal family of neighboring zeros respectively ones.  A zero block
is \emph{inner}, if in $\omega$ there are ones before and after it.
For any finite sequence $\omega\in\{0,1\}^{n+1}$ for some $n\in \NN$
set
\begin{align*}
N(\omega)&:=\#\{j~:~\omega_j=1\}, \notag \\
K(\omega)&:=\#~~ \textrm{inner blocks of 0 digits of odd
length}, \\
\nu(\omega)&:= \prod_{i=0}^n (1-\omega_i) =\begin{cases} 1
&\textrm{if}~~ \forall ~~ \textrm{digits of} ~~\omega~~\textrm{are}
~~0
\\ 0 &\textrm{if}~~ \exists ~~ \textrm{ digit of}
~~\omega~~\textrm{which is} ~~1.
\end{cases} \notag
\end{align*}
\end{definition*}
So, e.g., if $n=10$ and
$\omega=(1,0,0,1,0,1,1,0,0,0,1)\in\{0,1\}^{11}$, then $N(\omega)=5$
and $K(\omega)=2$, and for
$\omega=(0,1,0,0,1,0,1,1,0,0,0,1)\in\{0,1\}^{12}$ the same values of
$N$ and $K$ occur, too. For both cases $\nu(\omega)=0$. Note that in
view of the condition that $K$ is the number of \emph{inner} zero
blocks of odd length, we must have $N(\omega)>K(\omega)$, unless
both numbers are zero.

\begin{lemma}\label{l:bdeg} We have
$\tau(\omega)=n-1+N(\omega)-K(\omega)+\nu(\omega)$. Moreover, there
is always an extremal polynomial of minimal degree $\deg
p=\tau(\omega)$, so that $p$ is the product of linear factors
vanishing at the nodes $t_i$ ($i=1,\dots,n$) only.
\end{lemma}
\begin{proof}
\resstep{} \step{}  First of all, consider the case of $n=1$, when
there is only one node $t_1$, where $p$ must vanish. That entails
that for any polynomial $p$ satisfying the conditions of
\eqref{eq:tauomega}, $\deg p\geq 1$ and $(x-t_1)|p(x)$, in the sense
of polynomial division in $\RR[x]$. There is the case
$\omega=(0,0)$, when $p(x)=x-t_1$ suffices; and it also suffices for
$\omega=(0,1)$, too, while for $\omega=(1,0)$ the opposite, i.e.
$-(x-t_1)$ will suffice. Finally, no linear function, vanishing at
$t_1$, can remain positive on both sides of $t_1$, so $\tau(\omega)$
must be at least 2 if $\omega=(1,1)$: and indeed, $(x-t_1)^2$ is
appropriate. In all these cases $K(\omega)=0$, and it is easy to
check that the formula, stated by the lemma, holds true for $n=1$:
also, the given minimal degree polynomials are products of zero
factors vanishing at the only node $t_1$.

\step{} Let now $n=2$. Now $\tau(\omega)\geq 2$ and
$p_0(x):=(x-t_1)(x-t_2)|p(x)$. So we have $\tau(\omega)=2$ for
$\omega=(0,0,0)$, $\omega=(1,0,1)$, $\omega=(1,0,0)$ and
$\omega=(0,0,1)$, as $p_0$ itself shows, and also for
$\omega=(0,1,0)$, as shown by $-p_0$. On the other hand, with
exactly two neighboring $1$'s in the sequence $\omega$, $c\cdot p_0$
itself will not suffice, as it necessarily changes sign at $t_1$ and
at $t_2$, and only there. So at two intermediate points $x_j$ around
any of these nodes the sign of $c\cdot p_0$ is always the opposite.
On the other hand, $(x-t_1)(x-t_2)^2$ works for $\omega=(0,1,1)$,
and symmetrically $(x-t_1)^2(t_2-x)$ works for $\omega=(1,1,0)$, so
for these two sequences $\tau(\omega)=3$. At last, consider
$\omega=(1,1,1)$. We claim that $\tau(\omega)>3$, (and whence
$\tau(\omega)=4$, as shown by $(x-t_1)^2(x-t_2)^2$). But it is easy
to see that $q(x)=c(x-\xi)(x-t_1)(x-t_2)$ cannot work for any choice
of $\xi$ and $c\ne 0$, since $c(x-\xi)$ can not have alternating
signs at the points $x_0<x_1<x_2$. It follows $\tau((1,1,1))=4$.
Again, it is easy to check that the asserted formula holds for
$n=2$, and of course the above given extremal polynomials are just
products of factors of the form $(x-t_1)$ and $(x-t_2)$.

\step{} The very last case is easily generalized to the cases when,
for $n\in\NN$ arbitrary, our sequence of $\omega$ is the identically
$1$ sequence, $\omega={\mathbf 1}=(1,1,\dots,1)\in\{0,1\}^{n+1}$. We
assert that then $\tau({\mathbf 1})=2n$. The cases $n=1,2$ having
already been obtained, our proof is by induction. Let $n>2$ and
assume the assertion for all $n'<n$.

Let $p(x)$ be a polynomial in the definition \eqref{eq:tauomega} of
$\tau$. We distinguish two cases: first assume that there exists
some $x_j$ with $p(x_j)=0$. Then with any neighboring node -- say,
if $j\ne 0$, we can take $t_j$, otherwise $t_1$ for $x_0$ -- we can
consider the new polynomial
$\widetilde{p}(x):=p(x)/((x-t_j)(x-x_j))$, or
$\overline{p}(x):=p(x)/((x-t_j)(x-x_j))$ if $j=0$. Since these cases
are similar, we describe the case of $\widetilde{p}$ only. This
polynomial vanishes for all $t_i$ unless $i=j$, and is still
nonnegative at all $x_i$'s, save $i=j$, because all the other
$x_i$'s are outside the interval $(x_j,t_j)$, and thus the
denominator of $\widetilde{p}(x)$ remains positive there.  At $t_j$
and $x_j$ we neither know, nor we are interested any more in the
sign of $\widetilde{p}(x)$. But the above means that
$\widetilde{p}(x)$ is a proper polynomial for
$\widetilde{t}=(t_1,\dots,t_{j-1},t_{j+1},\dots,t_n)$ and
intermediate points $\widetilde{x}:=x(\widetilde{t})=
(x_0,x_1,\dots,x_{j-1},x_{j+1},\dots,)$ of length one less, and with
$\omega={\mathbf 1}\in\{0,1\}^n$ now. So the induction hypothesis
applies with $n'=n-1$, and $\deg \widetilde{p}\geq 2n'$, hence $\deg
p\geq 2n$.

Consider now the second case, when for all $i=0,\dots,n$, we have
$p(x_i)>0$. Then in the open interval $(x_i,x_{i+1})$ $p$ must have
at least two zeroes, (when zeroes are counted with multiplicity),
since it has at least one (at $t_i$), and only one simple zero does
not allow for the strictly positive values at both ends of the
interval. But that means that altogether $p$ has at least $2n$
zeroes, proving again that $\deg p\geq 2n$, as claimed.

On the other hand, $\prod_i(x-t_i)^2$ is obviously a proper
polynomial in the definition of $\tau(\omega)$ for all $\omega$,
(furthermore, it is the product of zero factors $(x-t_i)$, too)
hence $\tau({\mathbf 1})\leq 2n$, and thus $\tau({\mathbf 1})=2n$.

\step{} The above argument is also a model for our next
considerations: we argue by induction, reducing the number of digits
(the length) of $\omega$, whenever possible. Let $n\geq 3$, and
assume that $\omega$ contains some zero digit(s) (as $\omega=\mathbf
1$ has already been settled). Our reduction steps will concentrate
on the zero digit(s).

First, consider the case when $\omega$ has some zero at one of its
ends, say at the start of it. Then for $p$ any proper polynomial in
\eqref{eq:tauomega}, the polynomial $\widetilde{p}(x):=p(x)/(x-t_1)$
is a proper polynomial for the new system of nodes
$\widetilde{t}:=t\setminus \{t_1\}$ arising by removal of $t_1$ but
keeping all the $x_j$'s with $\omega_j=1$ (and thus replacing
$\omega$ by $\widetilde{\omega}$, the sequence with the first $0$
digit of $\omega$ removed). Note that the length decreased by one,
$N$ and $K$ remained unchanged, and the degree of $p$ was decreased
by $1$. Furthermore, we have removed only a zero from $\omega$,
hence $\nu(\omega)=1$ if and only if $\nu(\widetilde{\omega})=1$. So
the induction hypothesis  with $\widetilde{n}=n-1$ and with the
given system $\widetilde{t}$, $x(\widetilde{t})$ and
$\widetilde{\omega}$ yields the assertion. Remark that with a
minimal degree polynomial $\widetilde{p}$ for $\widetilde{t}$ and
$\widetilde{\omega}$ is a product of zero factors $(x-t_i)$, thus
the corresponding $p(x):=(x-t_1)\widetilde{p}(x)$ is also of the
this type, as asserted.

\step{} Second, consider the case when there are neighboring zeroes
anywhere in the sequence $\omega$. Let us pick up the first such
pair, say $\omega_j=0=\omega_{j+1}$ with $j$ minimal having this
property. Then define $\widehat{\omega}$ by deleting  these two
zeroes from the sequence $\omega$. The characteristics $\nu$, $N$
and $K$ do not change, while the length decreases by 2. Let $p$ be a
polynomial satisfying the conditions in \eqref{eq:tauomega}, and put
now $\widehat{p}(x):=p(x)/((x-t_j)(x-t_{j+1}))$. Let now
$\widehat{t}=(t_1,\dots,t_{j-1},t_{j+2}, \dots,t_n)$ (of length
$n-2$). Then $\widehat{p}$ is appropriate for $\widehat{t},
\widehat{\omega}, \widehat{x}$ in the sense of \eqref{eq:tauomega}
if and only if $p$ is appropriate for $t,\omega, x$ in the same
sense. Conversely, starting out from a polynomial $\widehat{p}$,
proper for $\widehat{t}$ etc., we find that
$p(x):=\widehat{p}(x)\cdot(x-t_j)(x-t_{j+1})$ will satisfy all
conditions for $t$ etc. This shows that
$\tau(\widehat{\omega})=\tau(\omega)-2$, as the minimal degree
polynomials correspond. That the minimal degree polynomials are the
product of zeros factors $(x-t_i)$ is again apparent.

\step{} Third, it remains to consider the case when there are no end
zeroes and neither there are double zeroes in $\omega$. That is, we
then must have a sequence of 1's with a few, say $k$, isolated
interior zeroes inside; let the set of indices with $\omega_j=0$ be
denoted by $z$. Then removing $\{t_j~:~j\in z\}$ from $t$, there
remains a set of $n-k$ nodes $\overline{t}:=\{t_i~:~1\leq i \leq n,
~\omega_i=1\}$, strictly interlacing with the set $x(t)$. Clearly,
$\tau(\omega)\geq \overline{\tau}(\omega)$ with
\begin{equation*}
\overline{\tau}(\omega):=\min\{\deg\overline{p}~:~
\overline{p}\not\equiv 0,~\overline{p}(x_j)\geq 0 ~(\omega_j=1),
p(t_i)=0 ~(i\notin z)\},
\end{equation*}
since in $\overline{\tau}(\omega)$ we have just relaxed a few
conditions in the definition \eqref{eq:tauomega}. Now we have a new
system of $\overline{t}, \overline{x}, \overline{\omega}={\mathbf
1}\in\{0,1\}^{n-k+1}$, for which a minimal degree polynomial is
known, according to the above settled case of $\omega=\mathbf 1$. We
thus have $\overline{\tau}(\omega)=\tau(\overline{\omega})=2(n-k)$,
hence in particular $\tau(\omega)\geq 2(n-k)$, and an extremal
polynomial for $\overline{\tau}(\omega)$ is
\begin{equation*}
\overline{p}(x):=\prod_{i\in\overline{t}} (x-t_i)^2.
\end{equation*}
However, if $\omega_j=0$, then $j\in z$ and, on the other hand,
$j+1\notin z$. Since $\omega_j=0$, in the interval $(t_j,t_{j+1})$
there is no condition on the sign of $p$ in \eqref{eq:tauomega}, and
outside of $[t_j,t_{j+1}]$ the function $(x-t_j)/(x-t_{j+1})$ is
strictly positive. Therefore, together with $\overline{p}$, also the
polynomial
\begin{equation*}
p(x):={\overline{p}(x)}{\prod_{\omega_j=0}\frac{x-t_j}{x-t_{j+1}}}=\prod_{\omega_j=0}
(x-t_j)(x-t_{j+1}) \prod_{\omega_{i-1}\cdot\omega_i=1}(x-t_i)^2
\end{equation*}
satisfies all sign conditions of \eqref{eq:tauomega}; observe, that
$p(t_i)=0$ now for all $i=1,\dots,n$. It shows that
$\tau(\omega)\leq 2(n-k)$, hence $\tau(\omega)=2(n-k)$.

It remains to observe that also in this case the asserted formula
holds true, since now $\nu(\omega)=0$, $N(\omega)=n-k+1$ and
$K(\omega)=k$; moreover, $p$ is again a product of zero factors with
the nodes as roots.
\end{proof}

\begin{corollary}\label{c:omegaandprime} Let $\Omega\subset \RR$ be
arbitrary, $t_1,t_2,\dots, t_n\in \Omega$ a node sequence, and
$\omega$ be the type of $t=(t_1,t_2,\dots,t_n)$ with respect to
$\Omega$. Then the minimal degree of polynomials, as defined in
\eqref{eq:bdegree} and \eqref{eq:bdegreeprime}, are both equal to
$\tau(\omega)$ in \eqref{eq:tauomega}.
\end{corollary}
\begin{proof} Clearly, the restriction $0\leq b|_\Omega$ is more stringent
than $0\leq b|_{\Omega'}$, equivalent to the definition of
$\tau(\omega)$ in \eqref{eq:tauomega}. But by Lemma \ref{l:bdeg}, at
least one $\tau$-extremal polynomial $p$ can be taken as product of
linear factors $(x-t_i)$, which does not vanish, hence does not
change sign, but only at nodes of $t$. Therefore, $p$ must be
strictly positive on the whole of $\cup_{\omega_j=1} (t_j,t_{j+1})$.
Putting $\Omega^{\star}:=\cup_{\omega_j=1} (t_j,t_{j+1}) \cup t$, we
are led to $p|_{\Omega^{\star}}\geq 0$, so in view of $\Omega\subset
\Omega^{\star}$, also $p|_{\Omega}\geq 0$. Whence \eqref{eq:bdegree}
cannot exceed $\tau(\omega)$, and the quantities \eqref{eq:bdegree}
and \eqref{eq:bdegreeprime} are equal.
\end{proof}

To give the relationship between the values  of $\tau(\omega)$ \and
$d_n(\Omega)$ the following considerations are appropriate. For a
given node-system we define
\begin{align*}
t^{(j)}&:=(t_1,\dots,t_{j-1},t_{j+1},\dots,t_n) \intertext{and then}
\omega^{(j)}:=\omega(t^{(j)})&\:=(\omega_0,\dots,\omega_{j-2},1,
\omega_{j+1}\dots,\omega_n)\in \{0,1\}^n,
\end{align*}
 since in the
interval $(t_{j-1},t_{j+1})$ there is always an element of $\Omega$,
namely at least $t_j$. This means that in $\omega$ we replace the
pair $t_{j-1},t_j$ by a single digit $1$ thus resulting in
$\omega^{(j)}$.

Given a node system $t$ (or equivalently $\omega$), first we have to
determine the maximal degree of the corresponding basis elements
$b_1,b_2,\dots, b_n$, and take the minimum with respect to the
choice of $t$. Also we have the restriction that $b_i$ ($i\ne j$)
vanishes at $t_j$ in a strictly larger order than $b_j$ (see the end
of Section \ref{subs:intr}). However, by the above we  know that the
minimal degree system of $b_j$'s can be constructed with zeroes only
in points of $t^{(j)}$. This means that by assuming strictly
positive values at $t_j$, we can forget about the extra assumption
on the order of vanishing. That is we first have to determine
\begin{equation*}
\sigma(\omega):=\max_{j=1,\dots,n} \tau(\omega^{(j)}),
\end{equation*}
which then will be the lowest possible maximal degree of the
$b_j$'s, given the set $\Omega$ and the node system $t\subset
\Omega$.

It is obvious that
$ d_n(\Omega)=\min_{t\subset \Omega, \#t=n}
\sigma(\omega)\quad\mbox{holds}$; 
see the formulation in \eqref{eq:dnomegadef}.

Before determining $d_n(\Omega)$ we first compute $\sigma(\omega)$
for any $0$-$1$ sequence $\omega$.
\begin{lemma}\label{l:sigma} For any $n\in\NN$ and
$\omega\in\{0,1\}^{n+1}$ the following formula holds:
\begin{equation}\label{eq:sigmaformula}
\sigma(\omega)=\begin{cases} n-1=\tau(\omega)-1 &\textrm{if} \quad
\omega={\mathbf 0}, \\
2n-2=\tau(\omega)-2 &\textrm{if} \quad \omega={\mathbf 1,} \\
2n-2=\tau(\omega)-1 &\textrm{if} \quad \omega=(0,1,\dots,1)
~~\textrm{or}~~
(1,\dots,1,0),\\
2n-3=\tau(\omega)-1 &\textrm{if} \quad \omega=(0,1,\dots,1,0), \\
\tau(\omega) &\textrm{otherwise}.
\end{cases}
\end{equation}
\end{lemma}
\begin{proof} The reader will have no difficulty in checking the
first four exceptional cases, hence we can skip the details.

Assume now that $\omega$ is not of the four exceptional case. Note
that since now $\omega\ne \mathbf 0$, we necessarily have
$\nu(\omega)=\nu(\omega^{(j)})=0$ for all $j=1,\dots,n$, hence by
Lemma \ref{l:bdeg}, for these $\omega$ we have
$\tau(\omega)=n-1+N(\omega)-K(\omega)$ and $\tau(\omega^{(j)})=n-2 +
N(\omega^{(j)}) - K(\omega^{(j)})$, in view of the decrease of the
length of the $0$-$1$ sequence. So actually have to keep track the
change of $N$ and $K$ when the $\omega^{(j)}$ are formed from the
given non-exceptional $\omega$.

Observe that $K(\omega^{(j)})\geq K(\omega)-1$, since no two
different (odd length) inner zero blocks can be influenced by the
operation of substituting $\omega_{j-1}$ and $\omega_j$ by one
single digit of $1$. Moreover, this decrease in the number of odd
length inner blocks can arise only if at one end of such a block,
the ending combination of $0$ and $1$ (or, $1$ and $0$) is replaced
by $1$. If any other, totally inner segment of an (odd) inner zero
block is replaced by $1$, then the number of zeros remaining from
the block remains odd, hence there will be at least one remaining
heir of the original odd block, again counted in $K(\omega^{(j)})$,
and thus not allowing $K$ to decrease. But if a pair of a $0$ and
$1$ is substituted by 1, then $N$ remains unchanged, so in all, in
case $K$ decreases, the difference of $N-K$ can increase only by $1$
(at most). Similarly, if $N$ increases, that requires that a pair of
zeros is substituted by $1$, so $N$ increases by $1$, but then $K$
can not decrease, as seen above. (It can increase, though, but that
is out of our concern now.) So, again, the difference $N-K$ can
increase at most by $1$. Summing up, we have, apart from the four
extra cases, $(N-K)(\omega^{(j)})\leq (N-K)(\omega)+1$.

We now show that for some particular index $j$ this unit increase of
$N-K$ really occurs. If $K(\omega)\geq 1$, then we just take an
ending pair of $0$ and $1$, as analyzed above, and replacing it by
$1$ will do the stated increase.

Finally, if there are no inner odd length zero blocks, then there
must be a pair of zeros somewhere. Indeed, either there is an even
inner block of zeros (in this case we find neighboring zeros), or
there is no inner block of zeros at all. But in the latter case
there must be at least two neighboring zeros at beginning or at the
end of $\omega$, since otherwise we would be in the exceptional
cases, which was excluded.

Also, $\omega\ne \mathbf 0$, so there must be some $1$'s in
$\omega$. So, let us consider an occurrence of a double zero in the
sequence $\omega$, which is neighboring a digit of $1$. Now replace
this two double zeros by $1$. Thus $N$ increases by $1$, while no
new inner block is generated moreover all odd inner zero blocks
remain odd inner zero blocks. Therefore we than have $N-K$
increasing by 1.

In all we find that apart from the four exceptional cases,
$(N-K)(\omega^{(j)})\leq (N-K)(\omega)+1$ for all $j=1,\dots,n$, and
for some particular choice of $j$ this inequality becomes an
equality: that is, according to the discussion above, we obtain
$\max_j\tau(\omega^{(j)})=\tau(\omega)$, as claimed.
\end{proof}

For low values of $n$ and for finite $\Omega$ the value of
$d_n(\Omega)$ is trivial to determine. These cases are included in
the next proposition for the sake of completeness only.

\begin{proposition}\label{p:sidecases} If $n=1$, then a
one dimensional positive basis of the constant function space
$\PP_0$ exists, for arbitrary choice of $t=\{t_1\}$. If $n=2$, then
$d_n(\Omega)=1$ for any (compact) $\Omega$ having at least two
elements, with the choice of $t=\{t_1,t_2\}=\{a,b\}$. If $\Omega$ is
finite, then $\dim C(\Omega) = \# \Omega$, so if $\#\Omega <n$, then
there is no nodes and basis systems of $n$ elements. If $\#
\Omega=n$, then $t=\Omega$ is the only choice, and
$d_n(\Omega)=n-1$. For all other cases $d_n(\Omega)>n-1$.
\end{proposition}
\begin{proof} Apart from the trivial cases, it suffices to show
that for $\# \Omega > n\geq 3$ we necessarily have
$d_n(\Omega)>n-1$. Indeed, for any choice of $t$, there remains a
point $x\in \Omega\setminus t$, hence the type $\omega$ of $t$ with
respect to $\Omega$ can not be $\mathbf 0$. We can now combine
Lemmas \ref{l:bdeg} and \ref{l:sigma}: from the latter, we are free
of the first exceptional case, and the next three give all  larger
values of $\sigma(\omega)$ (at least for $n\geq 3$).

In the last case of \eqref{eq:sigmaformula} we have $\sigma(\omega)=
\tau(\omega)$, and from Lemma \ref{l:bdeg} that $\tau(\omega)=
n-1+N(\omega)-K(\omega)$, since now $\nu(\omega)=0$. It remains to
see that $N(\omega)>K(\omega)$, which is clear if $K(\omega)>0$, and
also if $K(\omega)=0$ and $N(\omega)>0$, while
$N(\omega)=K(\omega)=0$ is excluded (for it leads to $\omega=\mathbf
0$). This concludes the proof.
\end{proof}

\begin{corollary}\label{c:noposbas} If $\#\Omega<n$, then the
$n$-dimensional polynomial space $\PP_{n-1}$ does not have a
positive basis in $C(\Omega)$: in particular, if $\Omega$ is
infinite, then no polynomial spaces $\Pm$ can have a positive basis
in $C(\Omega)$.
\end{corollary}

We consider now the case of $\Omega$ being still finite but
possessing sufficiently many points.
\begin{proposition}\label{p:fininf} If $3\leq n\leq \#\Omega<\infty$, then we
have
\begin{equation*}
d_n(\Omega)=\begin{cases}
n&\textrm{if}~~n<\#\Omega \\
n-1&\textrm{if}~~n=\#\Omega.
\end{cases}
\end{equation*}
\end{proposition}
\begin{proof} Proposition
\ref{p:sidecases}  gives the case of $n=\#\Omega$ and also
$d_n(\Omega)\geq n$ for  $n<\#\Omega$.

For the converse inequality let us take as $t_j$, for $j=1,\dots,n$,
the first $n$ points of $\Omega$, according to increasing order.
Then $\omega(t)=(0,0,0,\dots,0,1)$, as $b>t_n$ is a point of
$\Omega$. We have
$\sigma(\Omega)=n-1+N(\omega)-K(\omega)=n-1+1-0=n$.  So
$d_n(\Omega)\leq n$, the proof is complete.
\end{proof}

We have settled the case of a compact interval in Section
\ref{sec:interval}, and the above Proposition \ref{p:sidecases} and
\ref{p:fininf} contains the otherwise easy case of a finite
$\Omega$. Therefore, we now consider the remaining cases of a
compact, non-convex $\Omega$ with infinitely many elements. This is
actually the point where the topological-geometrical features of
$\Omega$ enter the picture. Since $\RR\setminus \Omega$ is open, it
is a union of disjoint open intervals: more precisely, $\RR\setminus
\Omega=(-\infty,a)\cup(b,\infty)\cup \left( \cup_{k\in\mathcal{K}}
I_k\right)$, where $\mathcal{K}$ is a finite or countably infinite
index set which is nonempty. We will use this decomposition of
$\RR\setminus \Omega$ as a canonical representation in the
following. Note that the intervals $I_k$ are disjoint and inside
$(a,b)$, but they might have common endpoints. We will denote the
endpoints of $I_k$ as $\alpha_k$ and $\beta_k$, so
$I_k=(\alpha_k,\beta_k)$.

We  call a family $\{I_{\ell}~:~\ell\in\mathcal{L}\}$ of holes
(i.e., intervals $I_k$, $k\in\mathcal{K}$) a \emph{free family of
holes}, if with the index set $\mathcal{L}\subset\mathcal{K}$ we
have $\overline{I_{\ell}}\cap \overline{I_{\ell'}}=\emptyset$,
($\ell\ne\ell', \ell,\ell'\in\mathcal{L}$).

Let $\dot{\Omega}$ denote the set of limit points of $\Omega$. There
might be some points in $\Omega\setminus\con \dot{\Omega}$. They are
all isolated points of $\Omega$, and we call these \emph{eccentric}
points of $\Omega$. Further, we define
\begin{equation*}
\lambda(\Omega):=\max \{ \#\mathcal{L} ~~:~~
\mathcal{L}\subset\mathcal{K} ~~\textrm{is a free family of
holes}\}.
\end{equation*}
Note that obviously $\lambda(\Omega)=\infty$ if and only if $\Omega$
has infinitely many components. This for example the case, if there
are infinitely many eccentric points. Whereas if there are only
finitely many of them, then we can introduce the notations
$\Theta_\ell$, $\Theta_r$ for the parity of their number on left
respectively on the right hand side of $\Omega$.

We now  describe $d_n(\Omega)$ in terms of these topological
characteristics.

\begin{proposition}\label{p:lambdainfty} If $\lambda(\Omega)=\infty$,
then we have
\begin{equation*}
d_n(\Omega)=\begin{cases} 1 &\textrm{if}~~n=1\\
1 &\textrm{if}~~n=2\\
n&\textrm{if}~~2<n.
\end{cases}
\end{equation*}
\end{proposition}
\begin{proof} As before, Proposition \ref{p:sidecases} settles the
cases $n=1,2$ and gives $d_n(\Omega)\geq n$ for $n\geq 2$, so again
it suffices to show $d_n(\Omega)\leq n$ for $n>2$.

By assumption we can take $n$ free holes, such that none of their
endpoints is $a$ or $b$. Arrange them according to the increasing
order, take every second of them starting with the first (altogether
$[n/2]$ free holes), and consider their endpoints as node system
$t$. That is, if the intervals are $I_1,\dots,I_n$, with
$I_j=(\alpha_j,\beta_j)$, ($j=1,\dots,n$), then consider the point
set $\{\alpha_1,\beta_1,\dots,\alpha_{2k-1},\beta_{2k-1},\dots\}$ of
$2[n/2]$ elements, and, in case $n$ is odd, add the point $b$ to
form $t$. For the corresponding type we have
\begin{align*}
\omega(t)&=(1,0,1,0,1,\dots,
1,0,1,0)\quad\mbox{if $n$ is odd,}\\
\mbox{or}\quad\omega(t)&=(1,0,1,0,1,\dots,1,0,1,0,1)\quad\mbox{if
$n$ is even}~.
\end{align*}
From this we see $K(\omega)=N(\omega)-1$, hence $d_n(\Omega)\leq
\sigma(\omega)=n-1+N(\omega)-K(\omega)=n$. By Proposition
\ref{p:sidecases} the equality $d_n(\Omega)=n$ follows.
\end{proof}

\begin{theorem}\label{l:simplelambda} Let $3\leq n$ and let $\Omega$ be an
infinite compact subset of $\RR$. Then we have
\begin{equation}\label{eq:dnsimple}
d_n(\Omega)=\begin{cases} n & \textrm{if} \quad \lambda(\Omega)>
\frac{n}{2},\\
n & \mbox{if } \lambda(\Omega)=\left[\frac{n}{2}\right], \mbox{ $n$ is even},\\
n+\Theta_\ell\Theta_r& \mbox{if } \lambda(\Omega)=\left[\frac{n}{2}\right], \mbox{ $n$ is odd},\\
2(n-1-\lambda(\Omega))+\Theta_\ell+\Theta_r &\mbox{if}\quad
\frac{n}{2}-1 \geq \lambda(\Omega)\geq 1,
 \\ 2n-3 &\textrm{if}\quad
\lambda(\Omega)=0.
\end{cases}
\end{equation}
\end{theorem}
\begin{proof}
The case $\lambda(\Omega)=0$ (the case of the interval) and
$\lambda(\Omega)=\infty$ is known by Section \ref{sec:interval} and
Proposition \ref{p:lambdainfty}. So we can assume $\lambda(\Omega)$
to be positive and finite. In this case $\Omega$ is the union of
finitely many closed (possibly degenerate) intervals, and the
interior of $\Omega$ is not empty. We will use Lemma \ref{l:bdeg}
and Lemma \ref{l:sigma} tacitly.

 \resstep{} \step{} First we consider the case
$\lambda(\Omega)>\frac{n}{2}-1$. The case of even $n=2k$ is easy.
Indeed, consider $k$ free holes $(\alpha_1,\beta_1),\dots,
(\alpha_k,\beta_k)$, ordered increasingly, and their endpoints as
nodes in $t$. Then
$\omega=\omega(t)=(\omega_0,0,\omega_2,0,\omega_4,0,\dots,\omega_{2k-2},0,\omega_{2k})$.
From which $K(\Omega)=N(\omega)-1$ and hence
$\sigma(\omega)=n-1+N(\omega)-K(\omega)=n$ is easy to see.
Proposition \ref{p:sidecases} then gives $d_n(\Omega)=n$.

\step{} Let $n=2k+1$. If $\lambda(\Omega)>k$ then we can take
similarly to the above $(\alpha_1,\beta_1),\dots,
(\alpha_k,\beta_k)$ free holes, such that none of the endpoints
coincides with $a$ (or $b$). In this case $t=\{a,
\alpha_1,\beta_1,\dots, \alpha_k,\beta_k\}$ (or with $b$ instead of
$a$) will be a proper node system for which (in case $\alpha_1\neq
a\in t$) we have
\begin{equation*}
\omega=\omega(t)=(0,\omega_1,0,\omega_3,0,\dots,\omega_{n-2},0,\omega_n).
\end{equation*}
So we have $K(\omega)=N(\omega)-1$, hence
$\sigma(\omega)=n-1+N(\omega)-K(\omega)=n$. As above $d_n(\Omega)=n$
follows.

\step{} For the case $n=2k+1$, $\lambda(\Omega)=k$ consider a
\emph{maximal} family of free holes with $k$ elements
$\{(\alpha_1,\beta_1),\dots,(\alpha_k,\beta_k)\}$, and index them
according to the increasing order.  If $\Theta_\ell=0$ we can shift,
if necessary, the holes having endpoints among the left eccentric
points, in such a way that $\alpha_1\neq a$. We can do the analogous
procedure for the eccentric points on the right of $\Omega$ (with
the condition $\beta_k\neq b$). Still we will have a free family of
holes having exactly $k$ members, which we will use in the
following. The set of the endpoints of these intervals is $t'$. If
$a\not\in t'$ (or $b\not\in t'$) we can add it to $t'$,
$t:=t'\cup\{a\}$, say. Such steps we \emph{cannot} accomplish if and
only if $\Theta_\ell\Theta_r=1$. In this latter case we take any
point which lies in the interior of $\Omega$ and add it to $t'$. In
both of these cases we have a node system of $n$ elements.
 We can easily determine $\omega=\omega(t)$
 \begin{align*}
\omega&=(0,\omega_1,0,\omega_3,\dots,\omega_{n-2},0,\omega_n),&\qquad\mbox{if
$\Theta_\ell\Theta_r=0$},\tag{a}\label{eq:a}\\
\omega&=(\omega_0,0,\omega_2,0,\dots,\omega_{2j},0,{1,1}_{\text{}}
,0,\omega_{2j+5},0,\dots, 0,\omega_n),&\qquad\mbox{if
$\Theta_\ell\Theta_r=1$}\tag{b}\label{eq:b}.
 \end{align*}
From this we can compute the value of $\sigma(\omega)$. In the case
\eqref{eq:a} we have $K(\omega)=N(\omega)-1$, hence
$\sigma(\omega)=n-1+N(\omega)-K(\omega)=n$. While in case
\eqref{eq:b} we obtain $K(\omega)=N(\omega)-2$, so
$\sigma(\omega)\leq n-1+N(\omega)-K(\omega)=n+1$. Summing up
$d_n(\Omega)\leq n+\Theta_\ell\Theta_r$. We leave the lower estimate
to the end of this proof. Actually the case $\Theta_\ell\Theta_r=0$
is already done, since Proposition \ref{p:sidecases} furnishes
$d_n(\Omega)\geq n$ hence $d_n(\Omega)=n$.

\step{} Now we turn to the case $\lambda\leq \frac n2-1$. The
arguments are very similar to the above.  Take a family of free
holes with $k=\lambda(\Omega)$ elements
$\{(\alpha_1,\beta_1),\dots,(\alpha_k,\beta_k)\}$ and index them
according to the increasing order. Repeat the shifting procedure for
the holes involving eccentric points as above (if possible).

If $\Theta_\ell=0$ we can shift, if necessary, the holes having
endpoints among the left eccentric points, in such a way that
$\alpha_1\neq a$. We can do the analogous procedure for the
eccentric points on the right of $\Omega$ (with the condition
$\beta_k\neq b$). Consider the $k'=2k+2-\Theta_\ell-\Theta_r$ points
$t':=\{a,\alpha_1,\beta_1,\dots,\alpha_k,\beta_k,b\}$. Note that,
for example, $a=\alpha_1$ if $\Theta_\ell=1$. Because $\Omega$ is
infinite and there are only finitely many holes in $\Omega$, there
is an interval $I$ completely belonging to $\Omega$. Take now $n-k'$
many points in the interior of $I$, this complements $t'$ to a node
system $t$ of $n$ elements.

The digits of $\omega$ are relatively well-known: first comes some
block  $\omega^{(\ell)}$ then a block of $n-k'+1$ many $1$'s and
finally a block $\omega^{(r)}$
\begin{equation*}
\omega=(\omega^{(\ell)},\underbrace{1,1,\dots,1}_{\text{ $n-k'+1$
ones}},\omega^{(r)}).
\end{equation*}
We observe only $\omega^{(\ell)}$ in detail; the analogous reasoning
applies also for $\omega^{(r)}$. If $\Theta_\ell=1$, then
\begin{equation*}
\omega^{(\ell)}=(0,0,\omega_2,0,\dots,\omega_{2j},0),
\end{equation*}
whereas if $\Theta_\ell=0$, we have
\begin{equation*}
\omega^{(\ell)}=(0,\omega_1,0,\omega_3,0,\dots,\omega_{2j-1},0).
\end{equation*}
 From this form we see that each $0$ block in $\omega^{(r)}$ is either an
inner block of zeros and in this case it is odd, or it is an ending
block and only then it might be of even length. Among
$\omega_2,\dots, \omega_{2j}$ (or $\omega_1,\dots,\omega_{2j-1}$)
there are $L$ digits of $1$. Denoting the number of $1$'s in
$\omega^{(r)}$ by $R$ and reasoning analogously as above we obtain
$N=n-k'+1+L+R$, and $K=L+R$, hence $\sigma(\omega)=n-1+N-K=2n-k'$
and
\begin{equation*}
d_n(\Omega)\leq 2n-2-2\lambda(\Omega)+\Theta_\ell+\Theta_r.
\end{equation*}

\step{}\label{case:5} Now we turn to the lower estimate of
$d_n(\Omega)$ in general. Let us now start with an optimal $t$ of
type $\omega=\omega(t)$, i.e., with $d_n(\Omega)=\sigma(\omega)$.
Denote $N:=N(\omega)$ and $K:=K(\omega)$. Denote by  $L$ and $R$ the
zeros at the left respectively the right side of $\omega$. Suppose
that there are $M$ inner blocks of zeros with respective lengths
$\mu_i$, $i=1\dots, M$. So $n+1-N-(L+R)= \sum_{i=1}^M \mu_i$. The
number of odd blocks is the number of odd $\mu_i$, so
$K=\sum_{i=1}^M \left([(\mu_i+1)/2]-[\mu_i/2]\right)$. Moreover, to
any zero-block of length, say, $\mu$, there corresponds a family of
$[(\mu+1)/2]$ free intervals. The union of any number of such
families, belonging to different zero-blocks, still remains a free
family of intervals. Therefore there exists a free family of holes
with $k:=\sum_{i=1}^M [(\mu_i+1)/2]$ members. So now we can write
\begin{align*}
2k& =2\sum_{i=1}^M \left[\frac{\mu_i+1}{2}\right]=\sum_{i=1}^M
\left(\left[\frac{\mu_i+1}{2}\right]+\left[\frac{\mu_i}{2}\right]\right)
+ K \\ &= \sum_{i=1}^M \mu_i + K =n+1 -N -(L+R)
+K=2n-(L+R)-\sigma(\omega).
\end{align*}
using again formula \eqref{eq:sigmaformula} and Lemma \ref{l:bdeg}.
Observe that $m$ zeros at the left (or right) side of $\omega$
result in additional family of $[m/2]$ free holes. Thus we can write
\begin{align*}
2\lambda(\Omega)&\geq 2k+2\Bigl[\frac{L}{2}\Bigr]+
2\Bigl[\frac{R}{2}\Bigr]= 2n-(L+R)+2\Bigl[\frac{L}{2}\Bigr]+
2\Bigl[\frac{R}{2}\Bigr]-\sigma(\omega)\\
&=2n-2+\Theta_\ell+\Theta_r-\sigma(\omega).
\end{align*}
By optimality of $t$  we have $\sigma(\omega)=d_n(\Omega)$. Thus
$d_n(\Omega)\geq 2n-2-2\lambda(\Omega)+\Theta_\ell+\Theta_r$, hence
the missing part of \eqref{eq:dnsimple} follows.

\step{} Only the case of odd $n$ and $\lambda(\Omega)=[n/2]$ remains
to be proved. We know the upper estimate $d_n(\Omega)\leq
n+\Theta_\ell\Theta_r $ from Step \ref{case:5} and the lower
estimate $d_n(\Omega)\geq
2n-2-2\lambda(\Omega)+\Theta_\ell+\Theta_r$ from Step 5. If
$\Theta_\ell\Theta_r=1$ these reduce to $d_n(\Omega)\leq n+1$ and
$d_n(\Omega)\geq 2n-2\lambda(\Omega)=n+1$, hence the assertion.
Otherwise we obtain $d_n(\Omega)=n$ by Proposition
\ref{p:sidecases}.
\end{proof}

Note that $\lambda(\Omega)$ can be determined in a rather mechanical
way through topological features. A way to do so is to consider all
boundary points -- of which there can be infinitely many if and only
if $\lambda(\Omega)=\infty$ and in this case we are done. Otherwise,
the finitely many boundary points determine finitely many intervals,
and we can form a finite sequence of $0$'s and $1$'s according to
the wether the intervals contain a point of $\Omega$ or not. If
$\mu_1,\mu_2,\dots,\mu_k$ are lengths of the inner zero-blocks, then
we have  $\lambda(\Omega)=\sum_{i=1}^k
\left[\frac{\mu_i+1}{2}\right]$.

For the subspaces $X\subseteq \Pm$ with positive basis and of
maximal dimension we obtain directly from Theorem
\ref{l:simplelambda} the following.
\begin{theorem} Let $\Omega$ be an infinite compact subset of $\RR$, $X\subseteq \Pm(\Omega)$ be a subspace with positive
basis and, as such, of maximal possible dimension. Then we have
\begin{equation*}
\dim X=\begin{cases} m & \textrm{if} \quad \frac{m}{2}\leq \lambda(\Omega),\\[1ex]
m-\Theta_\ell\Theta_r& \mbox{if } \frac{m-1}{2}=\lambda(\Omega),\\[1ex]
\bigl[\tfrac{m-\Theta_\ell-\Theta_r}{2}\bigr]+1+\lambda(\Omega)&\mbox{if}\quad
\frac{m-2}{2} \geq \lambda(\Omega)\geq 1,
 \\[1ex]
 \bigl[\tfrac{m+3}{2}\bigr] &\textrm{if}\quad
\lambda(\Omega)=0.
\end{cases}
\end{equation*}
\end{theorem}
Also in higher dimensions there are known (e.g., geometric)
conditions that imply the existence a positive basis in a given
finite dimensional subspace $X\subseteq C(\Omega)$. In connection to
our investigation we ask can pose the following
\begin{question}\label{q:multivariate} Consider $\Omega\subset\RR^d$
a compact, convex set -- or, a general compact set. What is the
maximal dimension of subspaces of $\Pn(\Omega)$ with positive
basis? 
\end{question}

\parindent0pt
\end{document}